\documentclass{amsart}
\usepackage{amssymb}
\usepackage{amsmath}
\usepackage{latexsym}
\input{xy}
\xyoption{all}

\newtheorem{theorem}{Theorem}[section]
\newtheorem{lemma}[theorem]{Lemma}

\newtheorem{Def}[theorem]{Definition}
\newtheorem*{KC}{Korselt's Criterion}
\newtheorem*{Lem}{Lehmer's Question}
\newcommand{\rad}[1]{\text{rad}(#1)}
\newcommand{\Z}{\mathbb{Z}}

\author[N. McNew]{Nathan McNew}
 \author[T. Wright]{Thomas Wright}
\address{Department of Mathematics\\ Towson University, 7800 York Road, Towson, MD 21204}
 \email{nmcnew@towson.edu}
\address{Department of Mathematics\\Wofford College, 429 N. Church St., Spartanburg, SC 29302}
 \email{wrighttj@wofford.edu}
\title{Infinitude of $k$-Lehmer numbers which are not Carmichael}

\begin{document}
\begin{abstract}
In this paper, we prove that there are infinitely many $n$ for which $rad(\varphi(n))|n-1$ but $n$ is not a Carmichael number.  Additionally, we prove that for any $k\geq 3$, there exist infinitely many $n$ such that $\varphi(n)|(n-1)^k$ but $\varphi(n)\nmid (n-1)^{k-1}$.  The constructs that we consider here are generalizations of Carmichael and Lehmer numbers, respectively, that were first formulated by Grau and Oller-Marc\'{e}n \cite{GOM}.
\end{abstract}
\maketitle

\section{Introduction}
In 1932, D.H. Lehmer \cite{Leh} asked the following question:

\begin{Lem} Do there exist positive composite integers, $n$, for which $\varphi(n)|n-1$?\end{Lem}

Here $\varphi(n)$ denotes the Euler totient function of $n$.  Note that any prime number trivially satisfies Lehmer's condition.  Composite $n$ satisfying this condition are said to be \textit{Lehmer numbers}, however no such numbers are known. If such a number, $n$ were to exist, it is known that $n$ would have to be greater than $10^{30}$ and that $n$ would have at least 14 prime factors \cite{CH}.  There is currently some disagreement as to whether one should expect Lehmer numbers to exist at all.

In an attempt to better understand the relationship between $\varphi(n)$ and $n-1$, Grau and Oller-Marc\'{e}n \cite{GOM} suggested the following weakening of the Lehmer condition:

\begin{Def} A composite natural number $n$ is called a \underline{$k$-Lehmer number} if $\varphi(n)|(n-1)^k$.  We denote the set of $k$-Lehmer numbers by $L_k$.\end{Def}

The idea of a Lehmer number can also be weakened in the following manner:

\begin{Def} Let $\kappa (n)=\rad{\varphi(n)}$, where $\rad{m}$ denotes the largest squarefree divisor of $m$.  We will call a composite natural number $n$ a \underline{radimichael number} if $$\kappa(n)|n-1.$$
\end{Def}
Note that an integer $n$ is a radimichael number if and only if it is a $k$-Lehmer number for some value of $k$.  The study of Lehmer numbers and their generalizations is closely connected to the study of Carmichael numbers, the pseudoprimes to the Fermat primality test:
\begin{Def} A composite number $n$ with the property that $a^n\equiv a$ (mod $n$) for every $a\in \mathbb Z$ is called a \underline{Carmichael number}.
\end{Def}
In 1899 Korselt \cite{Ko} gave an equivalent characterization of Carmichael numbers which makes their connection with Lehmer numbers more apparent. Define $\lambda(n)$ to be the Carmichael lambda function of $n$, the smallest integer such that $a^{\lambda(n)}\equiv 1$ (mod $n$) for every integer $a$ coprime to $n$. (Note that $\varphi(n)$ gives the order of the group $\Z/n\Z$, while $\lambda(n)$ gives the greatest order of any element of that group.)

\begin{KC}  A composite natural number $n$ is Carmichael number if and only if it is square-free and $\lambda(n)|n-1$.
\end{KC}

Let us adopt the notation that $\kappa(n)=\rad{\varphi(n)}$.  Since $\kappa(n)|\lambda(n)$, it is clear that any Carmichael number is also a radimichael number (hence the name).  However, the converse is not necessarily true; for example, 85 is a radimichael number that isn't Carmichael.  In this paper, we study a subset of the radimichael and $k$-Lehmer numbers which are not Carmichael.

\section{Results}
Because the radimichael numbers satisfy a substantially weaker condition than Carmichael numbers, it would be reasonable to expect there to be far more radimichael numbers than Carmichael numbers.  This expectation has not been borne out by heuristics or upper bounds, however.  The first author \cite{McN} proved that the best known upper bound for the count, $C(x)$, of the Carmichael numbers up to $x$,
\[C(x) \leq x^{1-(1+o(1))\log \log \log x/\log \log x }\]
as $x \to \infty$, holds for radimichael numbers as well.  Furthermore, Pomerance \cite{Po} argues that this bound is heuristically tight.  Regardless, it might be reasonable to expect that radimichael numbers would be easier to find than Carmichael numbers.

It has been known since 1994 \cite{AGP} that there are infinitely many Carmichael numbers.  From this, it follows trivially that there are also infinitely many radimichael numbers.  However, while it has been conjectured by the first author \cite{McN} that there are infinitely many radimichael numbers which are not Carmichael, this has not previously been demonstrated.  In this paper, using the results of Maynard and Tao on primes in tuples, we resolve this conjecture and in particular prove the following:





\begin{theorem} \label{thm:infrad}
There are infinitely many $n$ for which $\kappa(n)|n-1$ but $\lambda(n)\nmid n-1$.  Moreover, the number of radimichael numbers up to $x$ is at least $x^{\frac 12+o(1)}$ as $x \to \infty$.
\end{theorem}

For reference's sake, we note that the best known lower bound for the count of the Carmichael numbers up to $x$, due to Harmon \cite{Ha}, is $\gg x^{1/3}$.  This result then squares with our intuition that radimichael numbers should be easier to find using current mathematical technology.


Using our method of proof we are also able to give a more general result about $k$-Lehmer numbers.  The set of all radimichael numbers, $\bigcup_k L_k$, (where $L_k$ is the set of $k$-Lehmer numbers) is known to be infinite, since there are infinitely many Carmichael numbers.  However, the question of whether there are infinitely many $k$-Lehmer numbers for specific values of $k$ has been open.  Here, we resolve this question for nearly all values of $k$:


\begin{theorem} \label{thm:infkl}
For any $k\geq 3$, the set of $k$-Lehmer numbers which are not $(k-1)$-Lehmer, $L_k/L_{k-1}$ contains infinitely many integers with exactly $k-1$ prime factors.  In particular, the count of such numbers up to $x$ is at least $x^{\frac{1}{k-1}+o(1)}$ as $x \to \infty$.
\end{theorem}

We remark that as in the previous theorem, the $k$-Lehmer numbers that we construct are not Carmichael numbers.

In some sense, this result is optimal; as Grau and Oller-Marc\'{e}n \cite{GOM} proved that there are no 2-Lehmer numbers with exactly two prime factors, we know that it would be impossible to prove a general result about $k$-Lehmer numbers with exactly $k$ prime factors.  This result also highlights the fact that $k=2$ is a particularly interesting case; unlike 1-Lehmer numbers, it is known that 2-Lehmer numbers exist, but we are still unable to prove that there are infinitely many of them.





We make extensive use of the theorem of Maynard and Tao, combined with a method first introduced in \cite{WrI} to reach our result.  The trick is to use $k$-tuples chosen so that any primes in the tuple can be multiplied to produce a non-Carmichael radimichael number.  The second theorem is then simply a matter of counting exactly what sorts of radimichael numbers we have.



\section{Invoking the Maynard-Tao Theorem}
Recall that Maynard-Tao gives us the following:
\begin{theorem}  \label{MT} Let $a,b\in \mathbb N$ be such that $a\geq 2$, $b\geq 0$.  Consider the admissible $k$-tuple
$$D_{\{a,b,s\}}(n)=(a^{b+1}n+1,a^{b+2} n+1,a^{b+3}n+1,....,a^{b+s} n+1).$$
Then for any $m$, there exists a constant $H_{m}$ such that if $s\geq H_{m}$ then at least $m$ of the $s$ terms above are prime for infinitely many values of $n$.  In particular, the number of values of $n$ up to $x$ for which $D_{a,s}(n)$ contains at least $m$ prime numbers is $\gg_m \frac{x}{\log^s x}$.
\end{theorem}
The quantitative version stated here is proven in much greater generality in \cite{ma}.


We begin by showing that this result is sufficient to construct infinitely many radimichael numbers with any fixed number of prime factors; later, we will show that the numbers so constructed are not Carmichael.


\begin{lemma}\label{radm}
For any choice of $a,m \geq 2$, $b\geq 0$ let $s\geq H_{m}$, and let $n$ be such that the tuple $D_{\{a,b,s\}}(n)$ contains at least $m$ primes.  Label these primes $p_1<p_2<...<p_m$.  Then the quantity
$$N=p_1p_2...p_m.$$
is a radimichael number.
\end{lemma}
\begin{proof}
Let $p_1<p_2<...<p_m$ be as above.  Write
$$p_i=a^{l_i}n+1$$
for some $1\leq l_i\leq m$.  Then
$$\rad{p_i-1}=\rad{an}$$
for every $i$, and so
$$\kappa(N) = \rad{(p_1-1)(p_2-1)\cdots (p_m-1)} = \rad{an}$$
By construction, we have that
$$p_i\equiv 1\pmod{an},$$
which means that
$$N\equiv 1\pmod{an},$$
and so $\kappa(N)|N-1$.
\end{proof}

Because any Carmichael number must have at least 3 prime factors, this is already sufficient to prove Theorem \ref{thm:infrad}:
\begin{proof}[Proof of Theorem \ref{thm:infrad}]
Apply Lemma \ref{radm} in the case $a=2$, $b=0$, $m=2$.  For sufficiently large $x$, Theorem \ref{MT} implies that the count of the number of $n$ less than $\frac{x^{1/2}}{a^{H_2+1}}$ where at least two of the elements of $D_{\{2,0,H_2\}}(n)$ are prime is $\gg \frac{\sqrt{x}}{\log^{H_2} x}$.

For each such $n$, we use Lemma \ref{radm} to produce a radimichael number which is the product of two distinct primes, each less than or equal to $a^{H_2}n+1 \leq \sqrt{x}$.  Thus as $x \to \infty$ we obtain $x^{1/2+o(1)}$ radimichael numbers of size at most $x$.  Furthermore, none of these numbers are Carmichael as each has only two prime factors.
\end{proof}

\section{Radimichael but not Carmichael}

In this section, we prove that none of the radimichael numbers constructed using the method of Lemma \ref{radm} can be Carmichael numbers.  Since we can use our method to construct a radimichael number with $m$ prime factors for any value of $m\geq 2$, this then proves that there are infinitely many non-Carmichael radimichael numbers with any fixed number (at least two) of prime factors.

\begin{lemma} Let $N=p_1p_2....p_m$ be a radimichael number found by the method of Lemma \ref{radm}.  Then $N$ is not a Carmichael number.
\end{lemma}
\begin{proof}
Let $N$ be such a radimichael number, constructed from a tuple $D_{\{a,b,s\}}(n)$ containing $m$ prime factors.  As before, we will write
$$p_i=a^{l_i}n+1$$
with $l_1 < l_2 < \ldots<l_m$. By Korselt's Criterion, we know that if $N$ is a Carmichael number then
$$a^{l_i}n|N-1$$
for each $i$.  Specifically that this implies that
$$N\equiv 1\pmod{a^{l_2}n}.$$
Now, for each $i\geq 2$, $p_i\equiv 1$ (mod $a^{l_2}n$).  So
$$N\equiv p_1\equiv 1\pmod{a^{l_2}n}.$$
But this is impossible, since $a^{l_2}n>p_1$ and $N>1$.
\end{proof}

\par \noindent \textit{Remark}: We note that this proof is similar in structure to a proof that appears in \cite{WrI}, wherein the author uses this method to prove that there cannot exist a Carmichael number $N$ such that the prime factors of $N$ are all Fermat primes.








\section{$k$-Lehmer with $k-1$ factors}

Let us now turn our attention to categorizing these radimichael numbers by their $k$-Lehmer properties.  We are now able to give a proof of Theorem \ref{thm:infkl}. In particular we show that for each $k\geq 3$ there exist infinitely many $k$-Lehmer numbers (which are neither $(k{-}1)$-Lehmer numbers nor Carmichael numbers) with $k-1$ prime factors and that the count of such integers up to $x$ is at least $x^{\frac{1}{k-1}+o(1)}$ as $x \to \infty$.

\begin{proof}[Proof of Theorem \ref{thm:infkl}]
Let $m=k-1$ and fix $a \geq 2$ an integer.  For $b=mH_{m}$, consider the tuple
$$D_{\{a,b,H_m\}}(n)=(a^bn+1,a^{b+1}n+1,a^{b+2}n+1,....,a^{b+H_{m}}n+1).$$
Theorem \ref{MT} implies that there are infinitely many $n$ for which at least $m$ of these forms are simultaneously prime infinitely often, and using Lemma \ref{radm} we see that the product of these $m$ primes is a radimichael number.  Moreover, as in the proof of Theorem \ref{thm:infrad}, we find that there are $\gg_m x^{1/m+o(1)}$ values of $n$ where all of the resulting $m$ primes have size at most $cx^{1/m}$, where $c$ is chosen in such a way that the product of the primes has size at most $x$.  We will now see that the resulting radimichael numbers are in fact $k$-Lehmer numbers.

For a set of primes $p_1$,\ldots,$p_m$ coming from this tuple, write $p_i=a^{b+l_i}n+1$ where $l_i<l_j$ if $i<j$.  Let $N=p_1\cdots p_m$.  Then
$$\varphi(N)=\varphi(p_1)\cdots \varphi(p_m)=a^{l_1+\cdots +l_m+mb}n^m.$$
Note that $l_1\geq 1$ and $l_m\leq H_{m}$.  Additionally, $N-1$ can be expanded out as
\begin{equation}
N-1=a^{mb+l_1+l_2+\cdots +l_m}n^m+\cdots +\left(\sum_{i=1}^m a^{b+l_i}\right)n. \label{nminus}
\end{equation}

We show first that $\varphi(N)|(N-1)^{m+1}$.  It is clear that $a^{b+l_1}n$ divides every term of the expression above.  So
$$a^{b+l_1}n|N-1,$$
and hence
$$(a^{b+l_1}n)^j|(N-1)^j.$$
Now, if we let $j=m+1$ then
$$a^{l_1+\cdots+l_m+mb}n^m|(a^{b+l_1}n)^{m+1},$$
since $\sum_{i}l_i< mH_m = b$.  So $\varphi(N)|(N-1)^{m+1}$ and thus $N$ is a $k$-Lehmer number.

On the other hand, raising equation \eqref{nminus} to the $m$ and expanding, we can write
\[(N-1)^{m} = a^{mb+ml_1}n^{m}+ Y\]
where $a^{mb+ml_1+1}n^m|Y$. Thus $(N-1)^m$ is not divisible by $a^{mb+ml_1+1}n^m$ while $\varphi(N)$ is, so $N$ cannot be a $(k{-}1)$-Lehmer number.

\end{proof}

Finally, we note that in contrast to the construction here showing that there are at least $x^{\frac{1}{k-1}+o(1)}$ $k$-Lehmer numbers with $k-1$ prime factors, it was shown in \cite{McN} that the number of $k$-lehmer numbers up to $x$ is $\ll_k x^{1-\frac{1}{4k-1}}$, and that the number of radimichael numbers with $m$ prime factors is $\ll x^{1-\frac{1}{2m}}$.  In the special case that $m=2$ the count of radimichael numbers with 2 prime factors is less than \[x^{1/2}\exp\left\{\frac{2(2\log x)^{1/2}}{\log \log x}\left(1+O\left(\frac{1}{\log \log x}\right)\right)\right\}=x^{1/2+o(1)}\]
as $x \to \infty$. So, our lower bound for the count of radimichael numbers with two prime factors is nearly optimal.









\begin{thebibliography}{[AGP]}
\normalsize
\bibitem[AGP]{AGP} W. R. Alford, A. Granville, and C. Pomerance. \textit{There are infinitely many Carmichael numbers}, Ann. of Math. (2), \textbf{139} (3) (1994), 703-722.
\bibitem[Ca]{Ca} R. D. Carmichael, \textit{Note on a new number theory function}, Bulletin
of the American Mathematical Society, \textbf{16} (1910), 232-238.
\bibitem[CH]{CH} G. L. Cohen and P. Hagis, Jr., \textit{On the number of prime factors of $n$ if $\varphi(n)|n-1$}, Nieuw Arch. Wisk. (3) \textbf{28} (1980), no. 2, 177-185.
\bibitem[GOM]{GOM} J. M. Grau and A. M. Oller-Marc\'{e}n, \textit{On $k$-Lehmer numbers}, Integers \textbf{12} A37 (2012), 1081-1089.
\bibitem[Ha]{Ha} G. Harmon \textit{Watt's mean value theorem and Carmichael numbers.} Int. J. Number Theory \textbf{4}  (2008), 241-248.
\bibitem[Ko]{Ko} A. Korselt, Probl\`{e}me chinois, L'intermediaire des math., \textbf{6} (1899) 142-143.
\bibitem[Le]{Leh} D. H. Lehmer, \textit{On Euler's totient function}, Bull. Amer. Math. Soc. \textbf{38} (1932), no. 10,
745-751.
\bibitem[Ma]{ma} J. Maynard, \textit{Dense clusters of primes in subsets}, preprint, arXiv:1405.2593.
\bibitem[Mc]{McN} N. McNew, \textit{Radically weakening the Lehmer and Carmichael conditions},  Int. J. Number Theory \textbf{9} (2013), 1215-1224.
\bibitem[Po]{Po} C. Pomerance, \textit{Two methods in elementary analytic number theory}, Number theory and applications (Banff, AB, 1988), NATO Adv. Sci. Inst. Ser. C Math. Phys. Sci., vol. 265, Kluwer Acad. Publ., Dordrecht, 1989, pp. 135–161.
\bibitem[Wr]{WrI} T. Wright, \textit{The impossibility of certain types of Carmichael numbers}, Integers, \textbf{12} A31 (2012), 1-13.
\end{thebibliography}
\end{document}